\documentclass[11pt]{amsart}
\usepackage{hyperref}
\newcommand*\myat{{\fontfamily{pzc}\selectfont @}}
\usepackage[applemac]{inputenc}
\usepackage{amsmath,amssymb}
\usepackage{enumerate}

\newtheorem{thm}{Theorem}[section]
\newtheorem{Prop}[thm]{Proposition}
\newtheorem{lemm}[thm]{Lemma}

\theoremstyle{definition}

\newtheorem{rem}[thm]{Remark}

\newtheorem*{thm*}{Theorem}

\newcommand{\ce}{\mathbb E}

\newcommand{\Tr}{{\mathcal P}}

\newcommand{\sgn}{{\rm {sgn}}}
\newtheorem*{thank}{\ \ \ Acknowledgment}

\newcounter{tictac}

\def\1{\,\rlap{\mbox{\small\rm 1}}\kern.15em 1}

\def\build#1_#2^#3{\mathrel{\mathop{\kern 0pt#1}\limits_{#2}^{#3}}}
\def\tend#1#2{\build\hbox to 12mm{\rightarrowfill}_{#1\rightarrow #2}^{ }}

\def\converge#1#2#3#4{\build\hbox to
#1mm{\rightarrowfill}_{#2\rightarrow #3}^{\hbox{\scriptsize #4}}}

\newcommand{\Prod}{\mathop{\larger{\larger{\larger{\prod}}}}}

\def\Sup#1{\larger{\sup_{#1}}}

\def \h{h}

\newcommand{\beq}{\begin{equation}}
\newcommand{\eeq}{\end{equation}}

\def\M{\mathcal M}

\def\Z{\mathbb Z}
\def\Pr{\mathbb P}

\def\bR {{\mathbf b}\R }

\def\R{\mathbb R}
\def\T{\mathbb T}
\def\N{\mathbb N}
\def\C{\mathbb C}
\def\Q{\mathbb Q}

\def\AP {\mathcal AP}
\def \ds {\displaystyle}
\def \re {{\rm {Re}}}

\newtheorem{question}{Question}

\begin{document}

\title[Generalized Riesz Products on the Bohr compactification of $\R$]{Generalized Riesz Products
on the Bohr compactification of $\R^{(*)}$}

\author{E.\ H. \ {El} Abdalaoui}
\address{ Universit\'e de Rouen-Math\'ematiques \\
  Labo. de Maths Raphael SALEM  UMR 60 85 CNRS\\
Avenue de l'Universit\'e, BP.12 \\
76801 Saint Etienne du Rouvray - France.
}
\email{elhoucein.elabdalaoui@univ-rouen.fr \\ \href{mailto:foo@bar}{Gmail:elabdelh\myat gmail.com}}



\begin{abstract}
We study a class of generalized Riesz products connected to the spectral type of some class of
rank one flows on $\R$. Applying Kac's central limit theorem, we exhibit a large class of singular
generalized Riesz products on the Bohr compactification of $\R$.\\

$^*$Dedicated to Professors Jean-Paul Thouvenot and Bernard Host.\\
{\textbf {revised version II.}}
\vspace{8cm}\\

\hspace{-0.7cm}{\em AMS Subject Classifications} (2010) Primary: 42A05, 42A55;
Secondary: 11L03, 42A61.\\

\hspace{-0.7cm}{\em Key words and phrases:}
Generalized Riesz products, almost periodic functions,  Bohr compactification, mean value,  Kac's central limit theorem, Kakutani's criterion, Bourgain's singularity criterion, flat polynomials.\\
\end{abstract}
\maketitle
\newpage
\section{Introduction.}
The purpose of this paper is to extend and to study the notion of generalized
Riesz product in the setting of the Bohr compactification of $\R$.  This notion was formulated in the same manner as Peyri\`ere  in \cite{elabdal-mlem}. Therein, the authors proved that the spectral type of some class of rank one flows is given by some kind of generalized Riesz product on $\R$ normalized by some class of kernels. Here, following Peyri\`ere's suggestion  \cite{peyriere}, we give an alternative extension using the Bohr compactification of $\R$.\\

It is usual that the extension of some notions
from the periodic setting to the almost periodic ones requires
the consideration of the Bohr compactification of $\R$, which plays in the almost periodic case
the same role played by the torus $\T\stackrel{\rm {def}}{=}\{t \in \C,|t|=1\}$ in the periodic case.
As opposed to the torus, the Bohr compactification is often a non-separable compact topological space
and this lack of separability is a source of difficulties in trying to adapt the arguments from the periodic
context to the almost periodic one. Peyri\`ere also mentioned those difficulties in \cite{peyriere}.\\


Our analysis is also motivated by the recent growing interest in the problem of
the flat polynomials in the real setting suggested in \cite{prikhodko}.
It turns out that the main idea  developed in \cite{prikhodko} can not be adapted in ours. For a recent account on the problem of
the flat polynomials we refer the reader to \cite{QS}. \\

The paper is organized as follows. In section 2, we recall some standard facts on
almost periodic functions and the Bohr compactification of $\R$. In section 3, we introduce
generalized Riesz products on the Bohr compactification of $\R$ and state our main result concerning the singularity of a large class of these generalized Riesz products. In section 4, we summarize and extend the relevant material on Kakutani's criterion and Bourgain's criterion on the  singularity of the generalized Riesz products introduced in section 3.
In section 5, we state and prove the central
limit theorem due to M. Kac. Finally, in section 6, we apply Kac's central
limit theorem to prove our main result.

\section{The Bohr compactification of $\R$.}
The Bohr compactification of $\R$ is based on the theory of almost periodic functions initiated by
H. Bohr \cite{bohr} in connection with the celebrated Riemann's $\zeta$-function.
In this section we are going to recall the basic ingredients of this theory. For
the classical presentation we refer the reader to \cite{bohr}, \cite{Besicovitch} and
\cite[Chap V. section 2]{Katznelson}.

The space of all almost periodic functions is denoted by $\AP(\R)$. It is well known that $\AP(\R)$ is a subspace of the space of bounded continuous functions on $\R$. An important characterization of almost periodic functions is due to Bohr and it can be stated as follows

\begin{thm}[Bohr] A bounded continuous function $f$ is almost periodic function if, and only if, $f$ is uniformly
approximated by finite linear combinations of functions in the set
$\left\{\cos(tx),\sin(tx)\right\}_{t \in \R}$.
\end{thm}

\noindent{}We  will denote by $\bR$ the Bohr compactification of $\R$, and by $C(\bR)$ the space of continuous functions
 on $\bR$. We remind that one can define $\bR$ as follows:

\begin{thm}[\cite{Gelfand-RC}]\label{Gelfand}
The group $\R$, equipped with the usual addition operation, may be embedded as a dense subgroup of a compact abelian
group $\bR$ such that $\AP(\R)$ is the family of all restrictions functions $f_{|\R}$ to $\R$
of functions $f \in C(\bR)$. The operator $f \longmapsto f_{|\R}$ is an isometric $\star$-isomorphism
of $C(\bR)$ onto $\AP(\R)$.
Moreover, the addition operation $+~~:~~ \R \times \R \longrightarrow  \R$ extends
uniquely to the continuous group operation of $\bR$, $+~~ : \bR \times \bR \longrightarrow \bR.$
The group $\bR$ is called the Bohr compactification of $\R$.
\end{thm}

For simplicity of notation, for any $f$ in $\AP(\R)$, we use the same letter $f$ for its canonical
extension to $\bR$. We denote by $h$ the Haar measure on $\bR$, normalized to be a probability measure, and by $dt$ the usual Lebesgue measure on $\R$.\\

It is obvious that the continuous characters of $\bR$ are the
functions ${e^{i\lambda.}} ~:~\bR \longrightarrow \T$. For $f \in \AP(\R)$  we denote by $\int_{\bR} f(t) dh(t) $ the asymptotic mean value of $f$, given by
$$\int_{\bR} f(t) dh(t) =\lim_{T \longrightarrow +\infty}
\frac{1}{2T}\int_{-T}^{T}f(t) dt.$$

\noindent{}Following \cite{Katznelson}, for any $f\in \AP(\R)$, we put
\[
\widehat{f}\big(\{\lambda\}\big)=\int_{\bR} f(t) e^{-i\lambda t} dh(t).
\]
$\big(\widehat{f}\big(\{\lambda\}\big)\big)_{\lambda \in \R}$  are the Fourier coefficients of $f$ relative to orthonormal family
${\{e^{i \lambda t}\}_{\lambda \in \R}}$; the inner product is given by
$$<f,g>=
\int_{\bR} {f}(t) {\overline{g}}(t) d\h(t).$$\\
Furthermore, for any $p \geq 1$, we denote by $\|\cdot\|_p$ the norm in $L^p(\bR,h)$. We remind that  the sequence of probability measures $(\mu_n)$ on $\bR$ converge in the weak-$\star$ topology to some probability measure $\mu$ if, for any $\lambda \in \R$,
$$\widehat{\mu_n}(\lambda)=\int_{\bR}e^{-i \lambda t} d\mu_n(t) \tend{n}{+\infty}\widehat{\mu}(\lambda).$$

\section{Generalized Riesz Products on $\bR$.}
Riesz products was introduced in \cite{Riesz}. Therein, F. Riesz construct a continuous measure on the torus whose Fourier coefficients do not vanish at infinity. Roughly speaking, Riesz products are measures on $\T$ which are constructed inductively. This powerful construction can be used to produce examples of measures with desired properties. Later, Riesz's construction
was extended in \cite[p.208]{Zygmund}.\\

As shown in \cite{Ledrappier} Riesz products can be realized as a spectral type of some dynamical systems. Specific Riesz products are the right tool to describe the spectrum
of a class of dynamical systems arising from the
substitution \cite{Queffelec1},\cite{Queffelec2}.
A large class of Riesz products was realized in \cite{Host-mela-parreau} as the maximal spectral type of the unitary operator associated to a non-singular dynamical system and a cocycle over it. In \cite{Bourgain}, the author
established a connection between some class of generalized Riesz products
on the circle and the maximal spectral type of rank one maps. Alternative proofs were given
in \cite{Choksi-Nadkarni},\cite{Nadkarni} and \cite{Klemes-Reinhold}.\\

Generalized Riesz products analogous to  Peyri\`ere-Riesz products were realized in \cite{elabdal-mlem}  as
a spectral type of some class of rank one flows.\\

Here, our aim is to extend the notion of generalized Riesz products to $\bR$. Let $(p_k)_{k \in \N}$ be a sequence of positive integers greater than $2$ and $\Big((s_{k,j})_{j=0, \cdots,p_k}\Big)_{k \geq 0}$ be
a sequence of finite sequence of positive
real numbers with $s_{k,0}=0$ for any $k \in \N$. Put
\begin{eqnarray}\label{Polynomials}
P_k(t)=\frac1{\sqrt{p_k}}\sum_{j=0}^{p_k-1} e^{i t(jh_k+\overline{s_{k,j}})} ,
\end{eqnarray}
where $\overline{s_{k,0}}=0,$ and $\overline{s_{k,j}}=\sum_{l=0}^{j}{s}_{k,l},~1 \leq j \leq p_k-1$. The sequence $(h_k)$ is defined inductively by
\begin{eqnarray}\label{heigh}
h_0=1 {\rm {~~and~~}} h_{k+1}=p_k h_k +\sum_{l=0}^{p_k-1}{s}_{k,l}=p_kh_k+\overline{s_{k,p_k-1}}.
\end{eqnarray}
For  $p,q \in \{0,\cdots,p_k-1\}$, we introduce the following sequence of positive real numbers
\begin{eqnarray*}\label{spacers}
\overline{s_{k,p,q}}=\sum_{j=\min(p,q)+1}^{\max(p,q)}s_{k,j}=(\overline{s_{k,p}}-\overline{s_{k,q}})~\sgn(p-q).
\end{eqnarray*}
Finally, for any real number $t$, we set $e(t)=e^{it}.$

\begin{thm}[Generalized Riesz Products on $\bR$] Let $(P_n)_{n \in \N}$ be a family of trigonometric polynomials given by
 \eqref{Polynomials} and $\sigma_n=\ds {\Prod_{k=0}^{n}\big|P_k(t)\big|^2} \,d\h(t)$, $n=0,1,2,\cdots$. Then $(\sigma_n)_{n \geq 0}$ converge in the weak-$\star$ topology to some probability measure $\sigma$ on $\bR$. $\sigma$ will be denoted by
$$\ds {\Prod_{k=0}^{\infty}\bigg|P_k(t)\bigg|^2}.$$
\end{thm}

\begin{proof}{}
Let $R_n(t)=\big|P_{0}(t) \cdots P_n(t)\big|^2$. Then $\sigma_n= R_n(t)d\h(t)$, and by the
definition of $P_n$, we have
\[
\big|P_n(t)\big|^2=1+\Delta_n(t) {\textrm{~with~}}
\Delta_n(t) =\frac{1}{p_n}\sum_{p \neq q}
e\bigg(\big((p-q)h_n+\overline{s_{n,p,q}}~\sgn(p-q)\big)t\bigg).
\]
\noindent{}Hence, obviously
$$\int_{\bR} \big|P_n(t)\big|^2 dh(t)=1.$$
Put
$$W_n=\{(b-a)h_{n}+\overline{s_{n,b}}-\overline{s_{n,a}} \mid  b\neq a \in \{0,\cdots,p_{n}-1\}\}.$$
By expanding the product of the $ \left|P_{k}(t)\right|^2$, we can write
$$\int_{\bR} R_n(t) d\h(t)-1,$$
as a sum of terms of the type
$$\frac1{p_0 \cdots p_n}\int_{\bR} e\bigg( \sum_{j=1}^n \epsilon_jw_j t\bigg) dh(t),$$
where each $\epsilon_j$ is 0 or 1 (not all $=0$) and each $w_j$ belongs to $W_j$. To show that $\ds \int_{\bR} R_n(t) d\h(t)-1=0$,  it is sufficient to prove that each of these terms is null. For that, it is sufficient to prove that each of the numbers $\sum_{j=1}^k \epsilon_jw_j$ has absolute value $\geq1$.\\

We consider one of these expressions $\sum_{j=1}^n \epsilon_jw_j$ and denote by $j_0$ the greater index $j$ such that $\epsilon_j\neq0$. Then, we have
$$\left|\sum_{j }\epsilon_j w_j\right|\geq \left|w_{j_0}\right|-\sum_{j<j_0} \left|w_{j}\right|.$$
Moreover, for all $j$,
$$h_{j}\leq \left|w_{j}\right| \leq (p_{j}-1)h_{j}+\overline{s_{j,p_{j}-1}}.$$
Hence, to conclude, it is sufficient to prove that
$$
h_{{j_0}}-\sum_{j=0}^{j_0-1}\big((p_{j}-1)h_{j}+\overline{s_{j,p_{j}-1}}\big)\geq 1.
$$
The equality
$$
h_{n}=p_{n-1}h_{n-1}+\overline{s_{n-1,p_{n-1}-1}}
$$
implies by induction
$$
h_{n}-\sum_{k=0}^{n-1}(p_{k}-1)h_{k}+\overline{s_{k,p_{k}-1}}=1,
$$
whence
\begin{eqnarray}\label{proba1}
\int_{\bR} R_n(t) d\h(t)=1.
\end{eqnarray}
\noindent{}Thus $\sigma_n$ is a probability measure on $\bR$. In addition, for any $t \in \R$,  we have
\[
R_{n+1}(t)=R_n(t)|P_n(t)|^2=R_n(t)+R_n(t)\Delta_n(t),
\]
\noindent{}since $\widehat{R_{n}}$ and $\widehat{\Delta_{n}}$ are $\geq 0$ on $\R$
\begin{eqnarray*}
 \widehat{\sigma_{n+1}}\big(\lambda\big)&=&\widehat{R_{n+1}}\big(\{\lambda\})\\
 &=& \widehat{R_{n}}\big(\{\lambda\})+
\widehat{R_n}*\widehat{\Delta_n}(\lambda)\\
& \geq& \widehat{R_{n}}\big(\{\lambda\})=
\widehat{\sigma_n}\big(\lambda\big).
\end{eqnarray*}

\noindent{}Consequently the limit $r_{\lambda}$ of the sequence $(\widehat{\sigma_n}\big(\lambda\big))$ exists.
Now, since $\bR$ is a compact space, and $(\sigma_n)$ is a sequence of probability measures on $\bR$, we can
extract a subnet $(\sigma_{n_k})$ which converges weakly to some probability measure on $\bR$. This gives that the limit of $(\sigma_n)$ exists in the weak-$\star$ topology and the proof is complete.
\end{proof}
\noindent{}The proof above is largely inspired by Lemma 2.1 in \cite{Klemes-Reinhold}; it gives more,
namely, the polynomials $P_n$ given in \eqref{Polynomials} can be
chosen with positive coefficients and satisfying
$$\int_{\bR} \bigg|P_n(t)\bigg|^2 d\h(t)=1 {\rm {~~and~~}}
\int_{\bR} \Prod_{j=1}^{n}\bigg|P_j(t)\bigg|^2 d\h(t)=1.$$
We further mention that we have
\begin{eqnarray}\label{ries-prop}
\int_{\bR} \Prod_{j=1}^{k}\bigg|P_{n_j}(t)\bigg|^2 d\h(t)=1,
\end{eqnarray}
\noindent{}for any given sequence of positive integers
 $n_1<n_2<\cdots<n_k,~~k \in  \N^*$. The proof is the same as the proof of \eqref{proba1}.\\

We are now able to formulate our main result.

\begin{thm}[Main result]\label{main-here}
Let $(p_m)_{m \in \N}$ be a sequence of positive integers greater than $1$ and $((s_{m,j})_{j=0}^{p_m-1})_{m \in \N}$ be a
sequence of positive real numbers. Defining $(h_m)$ as in \eqref{heigh} and assuming that there exists a sequence of positive integers
$m_1<m_2<\cdots,$ such that, the numbers $h_{m_j},s_{m_j,0},\cdots,s_{m_j,{p_{m_j}-1}},$ $j=1,2, \cdots,$ are rationally
independent. Then the generalized Riesz product
$$\mu=\ds {\prod_{k=0}^{+\infty}\big|P_k(t)\big|^2}$$
where $P_k$ is given by \eqref{Polynomials},  is singular with respect to the Haar measure on $\bR$.
\end{thm}

We recall that the real numbers $\lambda_1,\lambda_2,\cdots,\lambda_r$, $r \geq 2$,
are rationally independent if they are linearly independent over $\Z$ , i.e. for all $n_1,\cdots,n_r \in \Z,$
$$n_1\lambda_1+\cdots+n_r\lambda_r=0 \Longrightarrow n_1=\cdots=n_r=0.$$

We say that the set of real number
 $\big\{\lambda_n,~~n \geq 1~\big\}$ is rationally independent if for any finite subsequence $m_1,m_2,\cdots, m_r$, the
 real numbers $\lambda_{m_1},\lambda_{m_2},\cdots,\lambda_{m_r}$ are rationally independent.
\section{On Kakutani's criterion and Bourgain's singularity criterion
in the setting of generalized Riesz products on $\bR$.}
The famous dichotomy theorem of Kakutani has a rather long history. In his 1948 paper \cite{kakutani},
Kakutani established a purity law for infinite product measures. More precisely, he proved that if
$\ds \Pr=\bigotimes_{i=1}^{+\infty}\Pr_i$
and $\ds \Q=\bigotimes_{i=1}^{+\infty}\Q_i$ are infinite product measures, where $\Pr_i, \Q_i$ are probability
measures such that $\Pr_i$ is absolutely continuous with respect to $\Q_i$, for each positive integer $i$. Then
\begin{center}
 $\Pr \ll \Q$  or $\Pr \perp \Q$ according as $\ds \prod_{i} \int \bigg(\frac{d\Pr_i}{d\Q_i}\bigg) d\Q_i$
converges or diverges.
\end{center}

There are a several proofs of Kakutani's criterion in literature (see \cite{Brown-Moran} and the references therein). For a proof based on the Hellinger's integral we refer the reader to \cite[p.60]{ZHellinger}.


Here, applying Bourgain's methods combined with the central limit tools introduced in \cite{elabdaletds}, we obtain a new extension of Kakutani's theorem in the setting of generalized Riesz products on $\bR$. Indeed, we will see that
the independence along subsequences allows us to prove the singularity.\\



We start by stating and proving the Bourgain's singularity criterion in the setting of $\bR$. We remind that this criterion in the periodic setting follows from the initial remarks in the proof of Proposition 1 in \cite[equations (2.15) and (2.22)]{Bourgain}.
\begin{thm}[$\bR$ version of Bourgain's criterion]\label{Bourg-cri}The following are equivalent
\begin{enumerate}[(i)]
\item $\ds \int_{\bR} {{\prod_{k=0}^{N}\big|P_k\big|\;}}  d\h \tend{N}{+\infty}0.$
\item $\mu$ is singular with respect to the Haar measure.

\item {\it $\inf \left \{\displaystyle  \displaystyle \int_{\bR}
{{\prod_{\ell=1}^L\big| {P_{n_\ell}}\big|\;}} d\h\;:\; L\in
{\N},~n_1<n_2<\ldots <n_L\right\}=0.$ }
\end{enumerate}
\end{thm}
\begin{proof}{}$(i)\Longrightarrow (ii)$: it suffices to show that
for any $\epsilon>0$, there is a Borel set $E \subset \bR$ with
$h(E)<\epsilon$ and $\mu(E^c)<\epsilon$. Let $0<\epsilon<1$, by $(i)$ there exists $N_0$ such that 
$$\ds \int_{\bR} {\prod_{k=0}^{N_0}|P_k|} \;d\h <\epsilon^2.$$ The
set $E=\left\{\lambda\in\bR\;:\;{\prod_{k=0}^{N_0}|P_k(\lambda)|} \geq\epsilon\right\}$ satisfies:
$$h(E)\leq \frac1\epsilon\left\|{\prod_{k=0}^{N_0}P_k}\right\|_1\leq \epsilon^2/\epsilon=\epsilon,$$
and, since $E^c$ is an open set, it follows by the portmanteau theorem that
\begin{eqnarray*}
 \mu(E^c) &\leq &\liminf_{M\to+\infty}\int_{E^c}{\prod_{k=0}^{M}|P_k|^2}\; d\h \\
& \leq & \liminf_{M\to+\infty}\int_{E^c}{
\prod_{k=0}^{N_0}|P_k|^2  \prod_{k=N_0+1}^{M}|P_k|^2}\;d\h \\
& \leq& \epsilon^2 \lim_{M\to+\infty} \int_{\bR}{\prod_{k=N_0+1}^{M}|P_k|^2}\; d\h=
\epsilon^2 <\epsilon.
\end{eqnarray*}

\medskip

For $(ii) \Longrightarrow (i)$: given $0<\epsilon<1$,
there exists a continuous function  $\varphi$  on $\bR$ such that:
$$0\leq\varphi\leq 1,\qquad
\mu(\{\varphi\neq 0\})\leq\epsilon \qquad (*)\qquad\hbox{and}\quad
h(\{\varphi\neq 1\})\leq\epsilon\qquad (**).$$
Indeed, by hypothesis, there exists a Borel set $E \subset \bR$ such that $\mu(E)=h(E^c)=0.$ Let $K \subset E$ be a compact set such that
$h(K^c) < \epsilon$. Since $\mu(K)=0$, we can choose a compact set $L \subset K^c$, such that $\mu(L^c) <\epsilon$. Then, by Urysohn's lemma \cite[pp.47]{Dudely}, there exists a continuous function $\varphi~~:\bR \rightarrow [0,1]$, taking the value $0$ at all points of $L$, and the value $1$ at all points of $K$. It is easy to check that $\varphi$ satisfy $(*)$ and $(**)$.\\

Let $f_N=\ds {\prod_{k=1}^N |P_k|}$. By Cauchy-Schwarz inequality, we have
\begin{multline*}
\int_{\bR} f_N\  d\h  =  \ds \int_{\{\varphi\neq 1\}} f_N\ d\h + \int_{\{\varphi= 1\}} f_N\ d\h
\\
 \leq h(\{\varphi\neq 1\})^{1/2}\left(\ds\int_{\bR} f_N^2\ d\h\right)^{1/2} + \left(\ds \int_{\{\varphi= 1\}} f_N^2\ d\h\right)^{1/2}
h(\{\varphi= 1\})^{1/2}
 \\\leq \sqrt{\epsilon} + \left(\ds \int_{\bR} f_N^2\, \varphi\ d\h\right)^{1/2}.
\end{multline*}
Since $\mu$ is the weak limit of $f_N^2 d\h$, we obtain
$$\lim_{N\rightarrow\infty} \int_{\bR} f_{N}^2\, \varphi\ d\h=\int_{\bR} \varphi\ d\mu
\leq \mu(\{\varphi\neq 0\})\leq\epsilon.$$
Thus,
$\limsup\ds \int_{\bR} f_N\ d\h \leq 2 \sqrt{\epsilon}$.
Since $\epsilon$ is arbitrary,  we get
$\ds \lim_{N\rightarrow\infty} \ds\int_{\bR} f_{N}\ d\h=0.$\\

Now $(i)$ obviously implies $(iii)$ and by  a simple application of Cauchy-Schwarz inequality, $(iii)$ implies $(i)$. Indeed, consider $n_1<n_2<\ldots <n_L$, $N\geq n_L$ and denote $\mathcal{N}=\left\{n_1<n_2<\ldots <n_L\right\}$ with $\mathcal{N}^c$
its  complement in $\left\{1,\cdots,N\right\}$.
\noindent{}Then
\begin{multline*}
\int_{\bR}\prod_{k=0}^{N}|P_k|\  d\h =
\int_{\bR} \prod_{k\in\mathcal N}|P_k|^{\frac12} \times \prod_{k\in\mathcal N^c}|P_k|^{\frac12}\prod_{k=0}^N|P_k|^{\frac12}\ d\h
\\\leq
\left(\int_{\bR}
\prod_{k\in\mathcal N}|P_k|\  d\h\right)^{\frac12}\left(
\int_{\bR} \prod_{k\in\mathcal N^c}|P_k|\times\prod_{k=0}^N|P_k|\  d\h\right)^{\frac12}
\\\leq
\left(\int_{\bR} \prod_{k\in\mathcal N}|P_k|\  d\h\right)^{\frac12}\left(
\int_{\bR} \prod_{k\in\mathcal N^c}|P_k|^2\  d\h\right)^{\frac14}\left(\int_{\bR} \prod_{k=0}^N|P_k|^2\
d\h\right)^{\frac14}.
\end{multline*}
Therefore
\begin{multline*}
 \int_{\bR}\prod_{k=0}^{N}|P_k|\  d\h \leq
\left(\int_{\bR} \prod_{k\in\mathcal N}|P_k|\  d\h\right)^{\frac12}\left(
  \int_{\bR} \prod_{k\in\mathcal N^c}|P_k|^2\  d\h\right)^{\frac14}\left(\int_{\bR} \prod_{k=0}^N|P_k|^2\
d\h\right)^{\frac14}\\=
\left( \int_{\bR} \prod_{k\in\mathcal N}|P_k|\  d\h(t)\right)^{\frac12}.
\end{multline*}

The last equality follows from \eqref{ries-prop}.
\end{proof}

From now, let us fix a sequence $\M$ of positive integers for which the set
\linebreak $\big\{h_m,s_{m,0},\cdots,s_{m,p_m-1},~~m\in \M \}$ is linearly independent over the rational numbers. Let $\ k \in {\N}$ and let ${\mathcal {N}}=\left\{
n_1<n_2<\ldots <n_k\right\}$ be a subsequence of $\M$. Put
\[
Q_k\left(t\right) =\Prod_{i=1}^k {P_{n_i}(t)}.
\]
Our strategy in the proof of our main theorem is to construct a subsequence \linebreak $\left\{
n_1<n_2<\ldots\right\}$ of $\M$ for which the Bourgain's criterion (Theorem \ref{Bourg-cri}) holds.\\

Having in mind applications beyond the context of this paper, we state and prove a sufficient condition for the existence of an absolutely continuous component with respect to the Haar measure for a given Riesz product on $\bR$. In the case of the torus, the result is due to M\'elanie Guenais \cite{Guenais}.
\begin{Prop}
If $\displaystyle
\sum_{k=1}^{+\infty}\sqrt{1-\bigg(\int_{\bR}{|P_k|\;}d\h\bigg)^2}<\infty$, then
$\mu$ admits an absolutely continuous
component.
\end{Prop}
\begin{proof}{}
Write $v_k^2=1-\|P_k\|^2$. Then
$\sum_{k=1}^\infty v_k < \infty$, equivalently  $\prod_{k=1}^\infty\|P_k\|_1 >0$. For all functions $f,g \in L^2(\bR, d\h)$, Cauchy-Schwarz inequality gives

$$\mid\| f\cdot g\|_1 - \|f\|_1 \|g\|_1\mid \leq (\| f\|_2^2 -\|f\|_1^2)^{1/2} (\|g\|_2^2 -\|g\|_1^2)^{1/2}.$$

Fix an integer $n_0 > 1$ and let $k > n_0$. Then

$$\bigg| \big\|\prod_{j=n_0}^k P_j\big\|_1 - \big\|\prod_{j=n_0}^{k-1}P_j \big\|_1\| P_k\|_1\bigg|$$
$$\leq \bigg(\big\| \prod_{j=n_0}^{k-1}P_j\big\|_2^2 - \big\|\prod_{j=n_0}^{k-1} P_j\big\|_1^2\bigg)^{1/2} \bigg(
\| P_k\|_2^2 - \| P_k\|_1^2\bigg)^{1/2}$$
$$\leq v_k.$$
Whence
$$\bigg| \big\|\prod_{j=n_0}^k P_j\big\|_1 - \big\|\prod_{j=n_0}^{k-1}P_j \big\|_1\| P_k\|_1\bigg| \leq v_k,$$
and this gives
\begin{gather*}
\bigg| \big\|\prod_{j=n_0}^{k-1} P_j\big\|_1\|P_k\|_1 -
\big\|\prod_{j=n_0}^{k-2}P_j \big\|_1\| P_{k-1}\|_1\| P_k\|_1\bigg| \leq v_{k-1}\|P_k\|_1 \leq
v_{k-1},\\
\begin{array}{ccc}
  \vdots & \vdots & \vdots \\
  \vdots & \vdots & \vdots\\
  \vdots & \vdots & \vdots\\
\end{array}\\
\bigg| \big\|\prod_{j=n_0}^{n_0+1} P_j\big\|_1\prod_{j=n_0+2}^{k}\|P_j\|_1 -
\prod_{j=n_0}^{k}\| P_j\|_1\bigg| \leq v_{n_0+1},
\end{gather*}

since, for any $j \in \N$,  $\|P_j\|_1\leq 1$  by \eqref{ries-prop}.
On adding the above inequalities:
$$\bigg| \|\prod_{j=n_0}^kP_j\|_1 - \prod_{j=n_0}^k\| P_j\|_1)\bigg| \leq \sum_{j=n_0}^kv_j.$$

Since $\prod_{j=1}^\infty \| P_j\|_1 > 0$ and $\sum_{j=1}^\infty v_k < \infty$,
we see that  $\ds \limsup_{k\rightarrow \infty}\| \prod_{j=1}^k P_j\|_1  > 0$. Whence,
by Bourgain's criterion for singularity (Theorem \ref{Bourg-cri}),  we see that $\mu$ is not singular to Haar measure
on $\bR$.
\end{proof}

\section{On Kac's Central Limit Theorem.}
Kac's central limit theorem in the setting of $\bR$ is stated and proved in \cite{Kac}.
For the sake of completeness we prove it here using standard probability arguments.



We will need the following classical central limit theorem \cite[p.81]{Dacunha-Castelle}. For our purpose, we state it in the following form
\begin{thm}[Complex CLT ]\label{MCLT} Let
$(Z_k)_{k \in \N}$ be a sequence of independent and identically distributed complex random variables. Suppose that  the real and imaginary parts of $Z_1$ have variance $\frac12$ and covariance $0$. Then, the sequence of random variables
$$
\left(\frac1{\sqrt {n}}\sum_{k=1}^{n} Z_{k}\right)_{n\geq1}
$$
converges in distribution to the complex gaussian measure $\mathcal N_{\C}(0,1)$ on $\C$.
\end{thm}

Now let us state and prove Kac's central limit theorem.
\begin{thm}[Kac's CLT \cite{Kac}]\label{Kac} Let $(\lambda_n)_{n \in \N}$ be a sequence of rationally independent real numbers.
Then the functions $\cos(\lambda_n t)+i\sin(\lambda_n t), n=1,\cdots$,
are stochastically independent and identically distributed under the Haar measure of $\bR$ and
converges in distribution to the complex gaussian measure $\mathcal N_{\C}(0,1)$ on $\C$.
\end{thm}
\begin{proof}By Theorem \ref{MCLT}, it suffices to prove that the random variables  $e(\lambda_n t)$, $n=1,\cdots$ are stochastically independent and identically distributed under the Haar measure of $\bR$. For that, notice that for any positive integer $k$ and for a given positive integers $l_1,l_2,\cdots, l_k$, we have
\begin{eqnarray*}
&&\int_{\bR}\bigg(e(\lambda_1 t)\bigg)^{l_1}
\bigg(e(\lambda_2 t)\bigg)^{l_2} \cdots \bigg(e(\lambda_k t)\bigg)^{l_k} d\h(t)
\\&&=
\int_{\bR}e\bigg(\big(\lambda_1l_1+\lambda_2l_2+\cdots+\lambda_kl_k\big) t\bigg) d\h(t)
\\&&=
\int_{\bR}\bigg(e(\lambda_1 t)\bigg)^{l_1}dt \int_{\bR}
\bigg(e(\lambda_2 t)\bigg)^{l_2}dt \cdots
\int_{\bR}\bigg(e(\lambda_k t)\bigg)^{l_k} d\h(t).
 \end{eqnarray*}
Indeed, $\ds \int_{\bR}e\big((\lambda_1l_1+\lambda_2l_2+\cdots+\lambda_kl_k) t\big) d\h(t)=1$ is equivalent to
$\lambda_1l_1+\lambda_2l_2+\cdots+\lambda_kl_k=0$, and by assumption, this is equivalent to $l_1=l_2=\cdots=l_k=0$. Therefore, if $\lambda_1l_1+\lambda_2l_2+\cdots+\lambda_kl_k=0$, then
$$\int_{\bR}e\big( (\lambda_1l_1+\lambda_2l_2+\cdots+\lambda_kl_k) t\big) d\h(t)=1
=\prod_{j=1}^{k}\int_{\bR}\bigg(e(\lambda_j t)\bigg)^{0}d\h(t),$$ and if
$\lambda_1l_1+\lambda_2l_2+\cdots+\lambda_kl_k \neq 0$, then
$$\int_{\bR}e\big((\lambda_1l_1+\lambda_2l_2+\cdots+\lambda_kl_k) t \big) d\h(t)=0=
\prod_{j=1}^{k}\int_{\bR}e(\lambda_j l_j t)d\h(t).$$
Hence, the random variables  $e(\lambda_n t),$ $n=1,\cdots$, are stochastically independent
under the Haar measure on $\bR$. It remains to prove that $e(\lambda_n t),$ $n=1,\cdots$, are identically distributed.\\

Let $s_1,s_2 \in \R$ and let us compute the characteristic function of the random variable $e( \lambda t)$, with
$\lambda \in \{\lambda_n,~~n=1,2,\cdots\}$. We have
\begin{eqnarray*}
\int_{\bR}e\bigg(\re\big((s_1+is_2).e(-\lambda t)\big)\bigg) d\h(t)
&=& \sum_{n=0}^{+\infty} \frac{i^n}{n!}\int_{\bR}\bigg(\re\big((s_1+is_2).e(-\lambda t)\big)\bigg)^n d\h(t)\\
&=& \sum_{n=0}^{+\infty} \frac{i^n}{n!}\int_{\bR}\bigg(s_1\cos(\lambda t)+s_2 \sin(\lambda t)\bigg)^n d\h(t).
\end{eqnarray*}
Write
$$\cos(\lambda t)=\frac{1}{2}\bigg(e(\lambda t)+e(-\lambda t)\bigg),~~~~
\sin(\lambda t)=\frac{1}{2i}\bigg(e(\lambda t)-e(-\lambda t)\bigg),
$$
and notice that we have
$$\int_{\bR}e(\alpha t)d\h(t)=
\lim_{T \longrightarrow + \infty}\frac{1}{2T}\int_{-T}^{T} e(\alpha t)dt=\begin{cases}
                                           1 {\rm {~~if~~ }} \alpha=0,\\
                                           0 {\rm {~~if~not. }}
                                          \end{cases}
$$
Therefore
$$\int_{\bR}\bigg(s_1\cos(\lambda t)+s_2 \sin(\lambda t)\bigg)^n d\h(t)$$
$$=\sum_{k=0}^{n}\binom{n}{k} \frac{(s_2+is_1)^k (-s_2+is_1)^{n-k}}{(2i)^n}\Big(\int_{\bR} e\big((2k-n) \lambda t\big) d\h(t)\Big).$$
Whence
$$\int_{\bR}\bigg(s_1\cos(\lambda t)+s_2 \sin(\lambda t)\bigg)^n d\h(t)
=\begin{cases}
\ds \binom{n}{\frac{n}{2}}(-1)^{\frac{n}{2}}\frac{(s_2^2+s_1^2)^{{\frac{n}{2}}}}{(2i)^n} &\text{if $n$~is~even,}\\
0&\text{if $n$~is~odd}.
\end{cases}
$$
We thus get
\begin{eqnarray*}
\int_{\bR}e\bigg(\re\big((s_1+is_2).e(-\lambda t)\big)\bigg) d\h(t)
&=&\sum_{l=0}^{+\infty} \frac{(-1)^l}{4^l(l!)^2}(s_1^2+s_2^2)^l=J_0(|s_1+is_2|)\\&=&
\frac1{\pi}\int_{0}^{\pi}\cos\bigg(|s_1+is_2|.\sin(t)\bigg) dt,
\end{eqnarray*}
where $J_0$ is the familiar Bessel function. We conclude that the hypotheses of central limit theorem \ref{MCLT} are satisfied, and  the sequence
$$
\left(\frac1{\sqrt {n}}\sum_{k=1}^{n} e(\lambda_k t)\right)_{n\geq1}
$$
converges in distribution to the complex gaussian measure $\mathcal N_{\C}(0,1)$ on $\C$. In particular,
by Theorem \ref{Chung} below,
\begin{eqnarray*}\label{Gauss2}
\int_{\bR}\Big|\frac1{\sqrt {n}}\sum_{k=1}^{n} e(\lambda_k t) \Big|d\h(t) \tend{n}{+\infty}\frac12{\sqrt{\pi}}.
\end{eqnarray*}
\end{proof}
\noindent{}It is well-known that the convergence in distribution or probability does not in general imply that
the moments converge (even if they exist). The useful condition to ensure the convergence of the moments
is the uniform integrability. Indeed, we have
\begin{thm}\label{Chung}If the sequence of random variables $\{X_n\}$ converges in distribution to some
random variable $X$ and if for some $p>0$, $\ds \sup_{n \in \N} \big(\ce(|X_n|^{p})\big) =M <+\infty$, then for each $r<p$,
\[
\lim_{n \longrightarrow +\infty}\ce\bigg(\big|X_n\big|^r\bigg)=\ce\bigg(\big|X\big|^r\bigg).
\]
\end{thm}
\noindent{}For the proof of Theorem \ref{Chung} we refer to \cite[p.32-33]{Billingsley} or
\cite[p.100]{Chung}. We remind that the condition
\begin{eqnarray*}\label{uniform}
\sup_{ n \in \N}\bigg(\ce\big(\big|X_n\big|^{1+\varepsilon}\big)\bigg) < +\infty,
\end{eqnarray*}
for some $\varepsilon$ positive, implies that $\{X_n\}$ are uniformly integrable.\\

Notice that our proof yields that for any sequence $(\lambda_j)_{j \in \N}$ of rationally independent
numbers, we have
\begin{eqnarray}\label{question}
 \int_{ \bR}\bigg|\frac{1}{\sqrt{q_n}}\sum_{j=0}^{q_n-1}e(\lambda_jt)\bigg|dt \tend{n}{\infty}
\frac12\sqrt{\pi}.
\end{eqnarray}

\section{Proof of the main result (Theorem \ref{main-here}).}
Using the analogue of F\'ejer's lemma combined with the CLT methods introduced in \cite{elabdaletds}, we will give a
direct proof of the singularity of a large class of generalized Riesz products on $\bR$. Therefore,
our strategy is slightly different from the strategy used by many authors in the case of the torus \cite{Bourgain}, \cite{Klemes},
\cite{Klemes-Reinhold}, \cite{elabdaletds}.
More precisely, they showed
that the weak limit point of
the sequence $\left(\left||P_m|^2-1\right|\right)$ is bounded below by a positive constant and it is well-known
that this implies the singularity of the generalized Riesz products (see for instance
\cite{elabdaletds} or \cite{Klemes}).\\

\noindent{}Let us start our proof with the following analogue of F\'ejer's lemma \cite[th 4.15, p.52]{Zygmund}

\begin{lemm}\label{Fejer} Under the notations of Theorem \ref{main-here}, let $n_1<n_2<\cdots<n_k \in \M$ and
$Q=\prod_{j=1}^{k}P_{n_j}$. Then, for any $m > n_k$, $m\in \M$, we have
\[
\int_{\bR} {|Q| \big|
{P_m} \big|} d\h = \bigg(\int_{{\bR}} {|Q|} d\h\bigg)
\bigg(\int_{\bR}{|P_m|} d\h\bigg).\]
\end{lemm}

\begin{proof} By our assumption $\big(h_k, (s_{k,p_k-1})\big)_{k \in \M}$ are rationally independent. Hence,
by Kac's theorem (Theorem \ref{Kac}), for $m>{n_k}$, $m\in \M$, $Q$ and $P_m$ are stochastically independent since
$Q$ is the sum of the random variables $e(\lambda t)$ with $\lambda \in \{jh_{n_i}+\overline{s_{n_i,j}}, j=0,\cdots,p_{n_i}-1,i=1,\cdots,k\}$ and $P_m$ is the sum of the random variables
$e(\lambda t)$ with $\lambda \in \{jh_{m}+\overline{s_{m,j}}, j=0,\cdots,p_{m}-1\}$.
\end{proof}


\begin{proof}[Proof of Theorem \ref{main-here}] By application of Lemma \ref{Fejer} combined with \eqref{question}, we proceed inductively to construct a subsequence $\{n_k\}$ of $\M$ such that, for any $k \geq 1$, we have
\begin{eqnarray}\label{geometric}
 \int_{\bR} { \Prod_{j=1}^{k+1} \big|{P_{n_j}}\big|} d\h
\leq \frac{51}{100}\sqrt{\pi}\int_{\bR} { \Prod_{j=1}^{k} \big|{P_{n_j}}\big|} d\h.
\end{eqnarray}
Denote $u_k=\ds \int_{\bR} { \Prod_{j=1}^{k} \big|{P_{n_j}}\big|} d\h$. Since $\ds \frac{51}{100}\sqrt{\pi}<1$, the sequence
$(u_k)$ will be decreasing with limit $0$. Hence, $\mu$ is singular with respect to the Haar measure by Bourgain's criterion (Theorem \ref{Bourg-cri}).\\

By our assumption combined with Theorem \ref{Kac}, $(P_m)_{m \in \M}$
converges in distribution to the complex gaussian measure $\mathcal N_{\C}(0,1)$. Hence, from Theorem \ref{Chung} applied with
$p=2$ and $r=1$ we get
\begin{eqnarray}\label{Gauss1}
\lim_{\overset{m\longrightarrow +\infty}{m\in \M}}\int_{\bR }{|P_m|}d\h=
\int_{\C} \big|z\big| d{\mathcal N_{\C}(0,1)}(z)=\frac{\sqrt{\pi}}{2}.
\end{eqnarray}
Remind that the density of the standard complex normal distribution $\mathcal N_{\C}(0,1)$ is given by
\[
    f(z) = \frac{1}{\pi} e^{-|z|^2}.
\]
Let us assume that we have already construct $n_1<n_2<n_3<\cdots<n_k$ satisfying $u_k
\leq \frac{51}{100}\sqrt{\pi} u_{k-1}$. By \eqref{Gauss1} there exists $m>n_k$, $m\in \M$,
such that
\[
\Big|\int_{\bR }{|P_m|}d\h-\frac{\sqrt{\pi}}{2}\Big| < \frac{\sqrt{\pi}}{100}.
\]
Therefore, by Lemma \ref{Fejer}, we can write
\[
\int_{\bR}{{|Q| |P_m|}} d\h \leq \frac{51\sqrt{\pi}}{100}
\int_{\bR}{|Q|} d\h,
\]
and set $n_{k+1}=m$. This implies that the inequality \eqref{geometric} holds
which finish the proof of the theorem.
\end{proof}
\begin{rem}The argument in the above
proof strongly depends on the assumption that
along subsequence the set of positive real numbers $\big\{h_m,s_{m,0},\cdots,s_{m,p_m-1},~~~~m\geq 1\big\}$ is
linearly independent over the rationales.
In the general case, one may use the methods of \cite{Bourgain}, \cite{Klemes}, \cite{Klemes-Reinhold}
and \cite{elabdaletds} to establish the singularity of a large
class of generalized Riesz products on
$\bR$, in particular when $(p_m)$ is bounded. In a forthcoming paper, we
 will see how to extend classical results from the torus and real line settings
to the Riesz products on $\bR$.
\end{rem}

In view of \eqref{question} one may ask
\begin{question} Let $n \in \N^*$ and put
\[
J_{\lambda,n}=\{P \in \Tr ~:~P(t)=\sum_{k=0}^{n}a_k e^{i \lambda_k t}, a_0=1,~a_k \in \{0,1\}\},
\]
where $ \Tr$ is the subspace of polynomials on $\bR$. The polynomials in $J_{\lambda,n}$ are called Newman polynomials. Is it true that have for any increasing sequence $(\omega_j)_{j \in \N}$ of real numbers, we have
 \[
\Sup{n \geq 1}\bigg\{\Sup{P \in J_{\omega, n}}\bigg(\frac{\big\|P(t)\big\|_1}
{\big\|P(t)\big\|_2}\bigg)\bigg\}<1? \]
\end{question}
In the contrast of the previous question, one may ask
\begin{question}
Does there exist a sequence of $L^1$-flat Newman polynomials on $\bR$? That is, does there exist a sequence
of polynomials $P_n \in J_{\lambda,n}$ such that
$$\lim_{n \longrightarrow +\infty} \frac{\big\|P_n(t)\big\|_1}
{\big\|P_n(t)\big\|_2}=1?$$
\end{question}

\begin{thank} The author would like to express thanks to anonymous referee for his/her helpful comments and suggestions, which significantly contributed to improve the quality of the paper.
He also express thanks to M. Lema\'nczyk, J-P. Thouvenot,  E. Lesigne,  A. Bouziad, J-M.
Strelcyn and  A. A. Prikhod'ko for fruitful discussions on the subject.
It is a pleasure for him to acknowledge the warm hospitality of the
Poncelet French-Russian Mathematical Laboratory in Moscow where a part of this work has been done.
\end{thank}

\end{document}